\documentclass[a4paper,11pt,oneside]{amsart}
\usepackage{header}

\title{On Ulrich bundles on projective bundles}

\author{Andreas Hochenegger}
\address{Politecnico di Milano, Dipartimento di Matematica ``F. Brioschi'', Italia}
\email{andreas.hochenegger@polimi.it}

\begin{document}

\maketitle

\begin{abstract}
In this article, the existence of Ulrich bundles on projective bundles $\bP(\cE) \to X$ is discussed. 
In the case, that the base variety $X$ is a curve or surface, a close relationship between Ulrich bundles on $X$ and those on $\bP(\cE)$ is established for specific polarisations.
This yields the existence of Ulrich bundles on a wide range of projective bundles over curves and some surfaces.
\end{abstract}

\section{Introduction}

Given a smooth projective variety $X$, polarised by a very ample divisor $A$,
let $i \colon X \into \bP^N$ be the associated closed embedding.
A locally free sheaf $\cF$ on $X$ is called \emph{Ulrich bundle} (with respect to $A$) if and only if it satisfies one of the following conditions:
\begin{itemize}
\item There is a linear resolution of $\cF$:
\[
0 \to \cO_{\bP^N}(-c)^{\oplus b_c} \to  \cO_{\bP^N}(-c+1)^{\oplus b_{c-1}} \to \cdots \to  \cO_{\bP^N}^{\oplus b_0} \to i_* \cF \to 0,
\]
where $c$ is the codimension of $X$ in $\bP^N$.
\item The cohomology $\sH^\bullet(X, \cF(-pA))$ vanishes for $1 \leq p \leq \dim(X)$.
\item For any finite linear projection $\pi \colon X \to \bP^{\dim(X)}$, the locally free sheaf $\pi_* \cF$ splits into a direct sum of $\cO_{\bP^{\dim(X)}}$.
\end{itemize}
Actually, by \cite{Eisenbud-Schreyer-Weyman}, these three conditions are equivalent.
One guiding question about Ulrich bundles is whether a given variety admits an Ulrich bundle of low rank.
The existence of such a locally free sheaf has surprisingly strong implications about the geometry of the variety, see the excellent surveys \cite{Beauville-Intro, Coskun-Intro}.

Given a projective bundle $\pi \colon \bP(\cE) \to X$, this article deals with the question, what is the relation between Ulrich bundles on the base $X$ and those on $\bP(\cE)$?
Note that answers to such a question depend much on the choice of a very ample divisor.
Here we consider very ample divisors $D = \pi^* A + H$, where $H$ is the relative hyperplane section of $\pi \colon \bP(\cE) \to X$.
In this case, we ask when a locally free sheaf of the form $\pi^* \cF (D)$ is an Ulrich bundle with respect to $D$.
By \autoref{prop:bundles-general}, this can be expressed solely in terms of cohomology vanishing on the base $X$.
If the base variety is a curve, then, essentially as a corollary, we obtain the following statement, see \autoref{thm:bundles-curves} and \autoref{rem:bundles-curves}.

\begin{theoremalpha}
\label{main:curve}
Let $\pi \colon \bP(\cE) \to C$ be a projective bundle over a smooth projective curve,
and let $D = \pi^* A + H$ be very ample.
Then a locally free sheaf $\pi^* \cF (D)$ is Ulrich if and only if $\sH^\bullet(C,\cF)=0$.

In particular, there are Ulrich line bundles on $\bP(\cE)$ with respect to $D$.
\end{theoremalpha}

Note that, for any very ample divisor $A'$ on $C$, the locally free sheaf $\cF$ gives rise to the Ulrich bundle $\cF(A')$ on $C$.
This result is a generalisation of \cite[Thm. 2.1]{Aprodu-etal}, where the existence of Ulrich line bundles on $\bP^1$-bundles over curves is shown.

If the base variety is a surface, then the connection between Ulrich bundles on $\bP(\cE)$ and those on the base becomes weaker.
There we have the following statement, see \autoref{thm:bundles-surfaces} and \autoref{cor:P2} \& \ref{cor:hirzebruch}.

\begin{theoremalpha}
\label{main:surface}
Let $\pi \colon \bP(\cE) \to S$ be a projective bundle over a smooth projective surface,
and let $D = \pi^*A + H$ be very ample.
Then a locally free sheaf $\pi^* \cF(D)$ is Ulrich if and only if 
\[
\sH^\bullet(S,\cF)=0 = \sH^\bullet(S, \cF \otimes \cO_S(-D')),
\]
where $D' = \rank(\cE) A + c_1(\cE)$.

In particular, there are Ulrich bundles of rank two with respect to $D$, if $S$ is the projective plane or a Hirzebruch surface.
\end{theoremalpha}

Note that if $D'$ is very ample on $S$, the locally free sheaf $\cF(D')$ from \autoref{main:surface} is an Ulrich bundle on $S$ with respect to $D'$.
So whenever there are Ulrich bundles known on the surface (for many polarisations),
then this implies the existence of an Ulrich bundle on $\bP(\cE)$.
By this we generalise results of \cite{Fania-etal} about the existence of Ulrich bundles on $\bP^1$-bundles over certain surfaces.
Note that the result here is not optimal: 
for example, under specific conditions on the Chern classes of a locally free sheaf $\cE$ of rank two,
the corresponding $\bP^1$-bundles over $\bP^2$ admits Ulrich line bundles, see \cite[Prop. 5.1]{Fania-etal}.

\marginpar{{\small The proof  of \autoref{prop:ulrich-bundle-pn} contains an error, see the Erratum on page \pageref{Erratum} for further details.}}
We also obtain the existence of Ulrich bundles on blowups of projective space in a linear subspace, using that such a blowup carries the structure of a projective bundle, see \autoref{prop:ulrich-bundle-pn}.

Finally we note that the method presented here may serve to obtain easily more Ulrich bundles on projective bundles over surfaces different from the projective plane or a Hirzebruch surface, see also \autoref{rem:further-surfaces}. 
But we also want to point out that the approach used here, does not generalise well for polarisations of the form $D = \pi^*A+nH$ with $n>1$.

\subsection*{Conventions}

Tensor products of a sheaf $\cF$ and a line bundle $\cO_X(D)$ are abbreviated as $\cF(D) \coloneqq \cF \otimes \cO_X(D)$.

For two sheaves $\cF$ and $\cG$ on a variety $X$, we  write $\Hom^\bullet(\cF,\cG)$ for the Hom complex
$\bigoplus_i \Ext^i_X(\cF,\cG)[-i]$,
which is a complex of vector spaces with zero differentials.
We use similar notation for the cohomology of a sheaf $\cF$, so $H^\bullet(X,\cF) = \Hom^\bullet(\cO_X,\cF)$.

Finally, all varieties here will be smooth and projective over an algebraically closed field of characteristic zero.

\subsection*{Acknowledgements}

The author wants to thank Klaus Altmann, Gianfranco Casnati, Francesco Malaspina, Peter Newstead, Joan Pons-Llopis and an anonymous referee for useful suggestions and comments.

\section{Preliminaries}

For the purposes of this work, we recall a cohomological characterisation of Ulrich bundles.

\begin{definition}
Let $X$ be a smooth projective variety
and $A$ a very ample divisor on $X$.
Then a locally free sheaf $\cF$ is called \emph{Ulrich bundle} with respect to $A$ if
\[
\sH^\bullet(X, \cF(-iA)) = 0
\]
for $i=1,\ldots,\dim(X)$.
\end{definition}

\begin{remark}
\label{rem:Serre-Ulrich}
If $\cF$ is an Ulrich bundle with respect to a very ample $A$ on a smooth projective variety $X$ of dimension $n$,
then also $\cF^\vee (K_X+(n+1)A)$ is Ulrich.
To see this note that by Serre duality
\[
\begin{split}
\sH^\bullet\bigl(X,\cF^\vee(K_X+(n+1-p)A\bigr) &=
\Hom^{n-\bullet}\bigl(\cF^\vee((n+1-p)A),\cO_X\bigr)^\vee = \\
& = \sH^{n-\bullet}\bigl(X,\cF((p-n-1)A)\bigr)^\vee
\end{split}
\]
and use that $\cF$ is Ulrich.
So whenever we construct an Ulrich bundle, we get automatically a second one.

Note that $\cF$ can coincide with $\cF^\vee (K_X+(n+1)A)$.
For example, if $\cF$ is of rank two, then $\cF^\vee \cong \cF(-c_1(\cF))$.
Hence for $\cF$ Ulrich with $c_1(\cF) = K_X + (n+1)A$, we find that
$\cF^\vee (K_X+(n+1)A) \cong \cF$.
If $X$ is a surface, then such Ulrich bundles are called \emph{special}; see also \cite{Eisenbud-Schreyer-Weyman, Casnati-q0}.
\end{remark}

\subsection{Projective bundles}

In the following, we collect some facts about projective bundles,
as can be found in \cite{Hartshorne}.

\begin{lemma}
\label{lem:bundles-canonical}
The canonical divisor of a projective bundle $\pi \colon \bP(\cE) \to X$ is given by
\[
K_{\bP(\cE)} = \pi^*(K_X + c_1(\cE)) - \rank(\cE)H,
\]
where $H$ is the relative hyperplane section.
\end{lemma}

Recall that $\pi$ is a flat morphism, so we do not run into problems pulling back divisors.

\begin{proof}
The relative canonical divisor for $\pi$ is given by
\[
K_{\bP(\cE)} \sim \pi^* K_X + K_{\bP(\cE)|X}.
\]
By \cite[Ex. III.8.4]{Hartshorne}, we have 
\[
K_{\bP(\cE)|X} \sim \pi^* c_1(\cE) - \rank(\cE)H,
\]
from which follows the statement.
\end{proof}

\begin{lemma}[{\cite[Prop. II.7.10]{Hartshorne}}]
\label{lem:bundles-veryample}
Let $\pi \colon \bP(\cE) \to X$ be a projective bundle over a smooth projective variety $X$
with relative hyperplane section $H$.
If $A$ is an ample divisor on $X$, then there is an $n > 0$ such that $\pi^*(nA) + H$ is very ample.
\end{lemma}

Finally, we give the main cohomological ingredient to establish much of the cohomology vanishing.

\begin{proposition}[\cite{Orlov}]
\label{prop:orlov}
Let $\pi \colon \bP(\cE) \to X$ be a projective bundle over a smooth projective variety $X$.
Then there is a semi-orthogonal decomposition
\[
\Db(\bP(\cE)) = \sod{ \pi^* \Db(X) \otimes \cO_\bP(-nH), \ldots, \pi^* \Db(X) \otimes \cO_\bP(-H), \pi^* \Db(X)},
\]
where $n+1 = \rank(\cE)$ and $H$ is the relative hyperplane section.
\end{proposition}

\begin{remark}
For readers which are unfamiliar with derived categories, we want to point out the two consequences of the semi-orthogonal decomposition above.
\begin{enumerate}
\item 
for coherent sheaves $\cF, \cG$ on $X$ we get
\[
\Hom^\bullet(\pi^*\cF, \pi^*\cG(-pH)) = 0
\]
for $p = 1, \ldots, n$. This is the semi-orthogonality of the decomposition, and is crucial for much of the cohomology vanishing that we will need.
\item 
any coherent sheaf on $\bP(\cE)$ has an `approximation' by objects in the parts $\pi^* \Db(X) \otimes \cO_{\bP(\cE)}(-pH)$.
Put the other way, we can build sheaves on $\bP(\cE)$ using exact triangles whose terms are in these parts. This will be used in \autoref{lem:XXX}.
\end{enumerate}
For more information on derived categories, see the introductory text \cite{Hochenegger} or the excellent text book \cite{Huybrechts}.
\end{remark}

\begin{example}
\label{ex:hirzebruch}
Just to see this in the simplest case of a Hirzebruch surface $\bF_r = \bP(\cE)$ with $\cE = \cO_{\bP^1} \oplus \cO_{\bP^1}(r)$.
There we have two natural candidates for the relative hyperplane section, namely $C_\pm$ with the property that $C_\pm^2 = \pm r$.
Using a fibre $f$, we have that $C_+ \equiv rf + C_-$ and $\omega_{\bF_r} = \cO(-(r+2)f -2C_-) = \cO((r-2)f - 2C_+)$.
Here we find that $\pi^* K_{\bP^1} = -2f$ and $\pi^* \det(\cE) = \cO(rf)$.
Plugging all this into \autoref{lem:bundles-canonical}, we conclude that $H = C_+$.
Note that a divisor $D = af+bC_+$ on $\bF_r$ is very ample if and only if $a,b>0$.

Finally we remark that \autoref{prop:orlov} is indifferent about the choice between $C_-$ and $C_+$. For this note that $\cO_{\bF_r}( C_+ ) = \cO_{\bF_r}( C_- ) \otimes \pi^* \cO_{\bP^1}(r)$ and $\blank \otimes \pi^* \cO_{\bP^1}(r)$ is an autoequivalence of $\pi^*\Db(\bP^1)$.
\end{example}

\section{Projective bundles over arbitrary varieties}

In this section, we treat the case of Ulrich bundles on projective bundles over arbitrary varieties. The following statement will be central for the next sections, where we specialise to varieties of low dimension.

\begin{proposition}
\label{prop:bundles-general}
Let $\pi \colon \bP(\cE) \to X$ be a projective bundle over a smooth projective variety $X$,
where $\cE$ has rank $n+1$.
For a very ample divisor $D = \pi^* A + H$ with $H$ the relative hyperplane section and
a locally free sheaf $\cF$ on $X$ with $\sH^\bullet(X,\cF) = 0$,
the locally free sheaf $\pi^* \cF(D)$ is Ulrich if and only if
\[
\Hom^\bullet_{X}\bigl(\sym^k \cE, \cF(-c_1(\cE)-(n+1+k)A)\bigr) = 0
\]
for $k=0,\ldots,\dim(X)-2$.
\end{proposition}

This proposition generalises \cite[Thm. 2.1]{Fania-etal}, where essentially this characterisation is given for Ulrich line bundles on projective bundles.

\begin{proof}
We abbreviate $\bP \coloneqq \bP(\cE)$ in the following.
By \autoref{prop:orlov}, we have a semi-orthogonal decomposition
\[
\Db(\bP) = \sod{ \pi^* \Db(X) \otimes \cO_\bP(-nH), \ldots, \pi^* \Db(X) \otimes \cO_\bP(-H), \pi^* \Db(X)}.
\]
Let $\cF$ be a locally free sheaf on $X$ with $\sH^\bullet(X,\cF) = 0$.
So by the semi-orthogonal decomposition above we find that
\[
\sH^\bullet(\bP,\pi^* \cF (-iD)) = 0
\]
for $i=0,\ldots,n$. 
Note that here $\dim(\bP) = n+m$, where $m = \dim(X)$.
So this is not sufficient to conclude that $\pi^*\cF(D)$ is an Ulrich bundle.
We need additionally the vanishing of 
\[
\sH^\bullet\bigl(\bP,\pi^*\cF(-(n+1+k)D)\bigr).
\]
for $k=0,\ldots,m-2$.
By \autoref{lem:bundles-canonical} we get
\[
-(n+1)H = K_{\bP} - \pi^*(K_X + c_1(\cE))
\]
Using this for the above cohomology group we see that
\[
\begin{split}
&\sH^\bullet(\bP,\pi^* \cF(-(n+1+k)D)) = 
\\ &=
\Hom^\bullet_{\bP}\bigl(\cO_{\bP}, \pi^*(\cF(-K_X-c_1(\cE)-(n+1+k)A)) \otimes \cO_{\bP}(K_{\bP}-kH)\bigr) =
\\ &=
\Hom^\bullet_{\bP}\bigl(\pi^*(\cF(-K_X-c_1(\cE)-(n+1+k)A)), \cO_{\bP}(kH)[n+m]\bigr)^\vee =
\\ &=
\Hom^\bullet_{X}\bigl(\cF(-K_X-c_1(\cE)-(n+1+k)A), \pi_*\cO_{\bP}(kH)[n+m]\bigr)^\vee =
\\ &=
\Hom^\bullet_{X}\bigl(\sym^k \cE, \cF(-c_1(\cE)-(n+1+k)A)[n]\bigr)
\end{split}
\]
where we use Serre duality two times, adjunction and $\pi_* \cO_\bP(kH) = \sym^k \cE$; see \cite[Prop. II.7.11]{Hartshorne}.
Note that these equalities are only valid on the level of derived categories, but here $\pi^*(\blank)$ and $\pi_* \cO_{\bP}(kH)$ coincide with their derived counterparts.
From this follows the claim.
\end{proof}

\begin{remark}
\label{rem:split-bundle}
For a locally free sheaf $\cE$ that splits into line bundles, the additional cohomology vanishing can be checked more easily in examples.
Let $\cE = \bigoplus_{i=0}^n \cO(D_i)$ and $\pi \colon \bP(\cE) \to X$,
and moreover, let $D = \pi^*A + H$ be very ample.
By \autoref{prop:bundles-general}, the additional conditions on a locally free sheaf $\cF$ with $\sH^\bullet(X,\cF)=0$ to be Ulrich amount to the vanishing of
\[
\sH^\bullet\Biggl(X, \cF\Bigl(-\sum_{i=0}^n D_i - \sum_{j=1}^k D_{i_j} -(n+1+k)A\Bigr) \Biggr)
\]
for all $0 \leq i_1 < \cdots < i_k \leq n$ and $k = 0,\ldots,\dim(X)-2$.

A result based on this remark is \autoref{prop:ulrich-bundle-pn}.
\end{remark}

\begin{remark}
Let $\pi \colon \bP(\cE) \to X$ be a projective bundle and $D = \pi^* A + H$ very ample.
For $\cG$ an Ulrich bundle with respect to $D$, one might ask whether there is a locally free sheaf $\cF$ on $X$ with $\sH^\bullet(X,\cF) = 0$, such that $\cG = \pi^* \cF(D)$.

In the following, we use knowledge about Ulrich line bundles on Hirzebruch surfaces, see \autoref{ex:hirzebruch-ulrich} for details.
First note that for these Ulrich line bundles, the answer to the above question is positive.
Now, consider the polarisation $D = f + C_+ = 2f+C_-$ for $\bF_1$.
There is a non-trivial extension of two Ulrich line bundles:
\[
0 \to \cO_{\bF_1}(C_+) \to \cG \to \cO_{\bF_1}(2f) \to 0.
\]
By \cite[Prop. 5.4]{Casanellas-Hartshorne},
$\cG$ is an indecomposable Ulrich bundle of rank two, and there is no other such locally free sheaf on $\bF_1$.
But $\cG$ cannot be of the form $\pi^* \cF(D)$, as all locally free sheaves on $\bP^1$ split into line bundles.
So the answer to the question above is negative in general.

Note that $\cG$ is an extension of the Ulrich line bundles $\cO_{\bF_1}(C_+)$ and $\cO_{\bF_1}(2f)$ for the same polarisation.
So we can ask whether all Ulrich bundles $\cG$ with respect to a polarisation $\pi^*A + H$ arise as extensions of Ulrich bundles of the shape $\pi^*\cF(D)$. But even this is not the case in general.

To see this, we consider a rational normal scroll $\bP(\cE) \to \bP^1$ where $\cE = \bigoplus_{i=0}^n \cO_{\bP^1}(a_i)$ and $a_i > 0$.
By \cite[Lem. 3.2]{Aprodu-HML}, the locally free sheaves $\Omega^j_{\bP(\cE)|\bP^1}(j,j+1)$ are stable Ulrich bundles with respect to $H$, for $j=0,\ldots,n$.
Hence, for $j>0$, these Ulrich bundles are neither obtained by pullbacks twisted by a divisor nor extensions thereof.
\end{remark}

\section{Projective bundles over curves}
\label{sec:curves}

In this section, we specialise to the question about Ulrich bundles on projective bundles over curves.

The following lemma is well-known. For the convenience of the reader, we give a proof.

\begin{lemma}
\label{lem:curves-immaculate}
Let $C$ be a projective curve of genus $g$.
Then $\sH^\bullet(C,\cL) = 0$ for a general line bundle $\cL$ of degree $g-1$.
\end{lemma}

\begin{proof}
Note that $\Pic^{g-1}(C)$ has dimension $g$, in particular, a general line bundle $\cL$ of degree $g-1$ is not effective, so $\sH^0(C,\cL)=0$.
The degree of $\cL$ was chosen in such a way that $\sH^1(C,\cL)=0$
by the Riemann--Roch theorem.
\end{proof}

\begin{remark}
Note that \autoref{lem:curves-immaculate} gives rise to many Ulrich line bundles on curves.
More generally, let $A$ be a very ample divisor and $\cE$ a locally free sheaf on $C$.
Then $\cE(A)$ is Ulrich with respect to $A$ if and only if $\sH^\bullet(C,\cE)=0$.

Hence the study of Ulrich bundles can be reduced to the search of locally free sheaves with no cohomology at all. 
Except for $C = \bP^1$, one can construct many indecomposable Ulrich bundles of higher rank;
for more details on the following classification, see \cite[\S 4.1]{Coskun-Intro}.

For $g=0$, that is, $C = \bP^1$, the only line bundle with no cohomology at all is $\cO_{\bP^1}(-1)$.
Recall that any locally free sheaf $\cE$ on $\bP^1$ splits into line bundles. So if $\cE$ is such that $\sH^\bullet(\bP^1,\cE)=0$, then it is isomorphic to $\bigoplus \cO_{\bP^1}(-1)$.

For $g=1$, we recall \cite[Thm. 5]{Atiyah}: for each $r \geq 1$, there is an indecomposable locally free sheaf $\cF_r$ of rank $r$ and of degree zero with $\sH^\bullet(C,\cF_r) \cong \kk$, additionally $\cF_r$ is unique up to isomorphism. Moreover, for $\cF_r \otimes \cL$ with $\cL$ a non-trivial line bundle of degree zero, we have $\sH^\bullet(C,\cF_r \otimes \cL) = 0$.
Hence, any locally free sheaf $\cE$ on an elliptic curve with $\sH^\bullet(C,\cE)=0$ splits into a direct sum of such $\cF_r \otimes \cL$.

For $g>2$, one can consider a direct sum of line bundles $\bigoplus \cL_i$ with $\sH^\bullet(C,\cL_i)=0$. Then a general deformation of this direct sum will yield an indecomposable locally free sheaf $\cE$ with $\sH^\bullet(C,\cE)=0$.
\end{remark}

As an application of \autoref{prop:bundles-general}, we get the following statement.

\begin{theorem}
\label{thm:bundles-curves}
Let $\pi \colon \bP(\cE) \to C$ be a projective bundle over a smooth projective curve $C$.
Moreover, let $D = \pi^* A + H$ be a very ample divisor with $H$ the relative hyperplane section. 

A locally free sheaf $\pi^* \cF(D)$ is Ulrich on $\bP(\cE)$ with respect to $D$ if and only if $\cF$ is a locally free sheaf on $C$ with 
$\sH^\bullet(C,\cF) = 0$.
In particular, for any very ample $A'$, $\cF(A')$ is Ulrich with respect to $A'$.
\end{theorem}

\begin{remark}
\label{rem:bundles-curves}
By \autoref{lem:bundles-veryample}, very ample divisors of the shape $\pi^* A + H$ exist on $\bP(\cE)$. 
Moreover, by \autoref{lem:curves-immaculate}, there are line bundles $\cL$ on $C$ such $\sH^\bullet(C,\cL) = 0$.
In particular, a projective bundle over a curve admits an abundance of Ulrich line bundles.
\end{remark}

\begin{remark}
The above result generalises \cite[Thm. 2.1]{Aprodu-etal}, where Ulrich line bundles on geometrically ruled surfaces over curves are classified.
Note that in \cite{Aprodu-etal}, also Ulrich bundles of rank two are constructed for any given polarisation of the ruled surface. This uses the Serre construction of locally free bundles, which does not generalise that easily to higher dimensions.

In \cite[Thm. 4.7, Cor. 4.10]{Aprodu-HML}, Ulrich bundles on rational normal scrolls $\bP\bigl( \bigoplus_i \cO_{\bP^1}(a_i) \bigr)$ with respect to polarisations of the form $\pi^*A +H$ are characterised. There it is shown that there are at most finitely many.

Moreover, in \cite[Thm. B]{Faenzi-Malaspina}, rigid Ulrich bundles on Hirzebruch surfaces of degree $\geq 5$ are described as extensions (more precisely: mutations) of Ulrich line bundles. This gives a curious connection between such Ulrich bundles and Fibonacci numbers.
Additionally, in the case that such a Hirzebruch surface is of degree $4$,
aCM sheaves of even rank are classified: there is a $\bP^1$-family of such sheaves, see \cite[Thm. A]{Faenzi-Malaspina}.

Finally, in \cite{Antonelli-Hirzebruch} a full classification of Ulrich bundles on Hirzebruch surfaces for any given polarisation is obtained.
\end{remark}

The following example can also be obtained as a consequence of the more general \cite[Thm. 2.1]{Aprodu-etal}.

\begin{example}
\label{ex:hirzebruch-ulrich}
Let 
$\bF_r = \bP(\cO_{\bP^1} \oplus \cO_{\bP^1}(r))$ be the $r$-th Hirzebruch surface with $r \geq 0$ and let $\pi \colon \bF_r \to \bP^1$ be the projection.
Up to isomorphism, the only line bundle $\cL$ on the base $\bP^1$ with $\sH^\bullet(\bP^1,\cL) = 0$ is $\cO_{\bP^1}(-1)$.

If we denote by $f$ a fibre of $\pi$ and by $C_+$ the rational curve on $\bF_r$ with $C_+^2 = r$,
then for $\cO_{\bF_r}(f) = \pi^* \cO_{\bP^1}(-1)$ we find that $\sH^\bullet(\bF_r, \cO_{\bF_r}(f)) = 0$.
Additionally, one can check that $D = af + C_+$ is ample, which for smooth projective toric varieties is equivalent to being very ample.
Hence \autoref{thm:bundles-curves} and \autoref{rem:Serre-Ulrich} imply that for this polarisation, the line bundles
\[
\cO_{\bF_r}((a-1)f+C_+) 
\quad\text{and}\quad
\cO_{\bF_r}((r-1+2a)f)
\]
are Ulrich.

For $r>0$ these give all Ulrich line bundles. To see this note that 
\[
\cO_{\bF_r}(-f),\ 
\cO_{\bF_r}(if - C_+) \text{ with $i \in \bZ$},\ 
\cO_{\bF_r}((r-2)f-2C_+),
\]
is a complete list of line bundles $\cL$ on $\bF_r$ with the property that $\sH^\bullet(\bF_r,\cL)=0$.
Therefore, the very ample divisor $D = af + bC_+$ with $a>0$ and $b>1$ does not give rise to any Ulrich line bundle.

For $r=0$, we find that $\bF_0 = \bP^1 \times \bP^1$ has precisely two families of line bundles $\cL$ with $\sH^\bullet(\bP^1\times\bP^1,\cL) = 0$:
\[
\cO_{\bF_0}(if-f'),\ 
\cO_{\bF_0}(-f+if') \text{ with $i \in \bZ$}
\]
where $f' = C_+$ is a fibre of the second ruling.
Hence given any very ample divisor $D = af+bf'$ with $a,b>0$, we get that 
\[
\cO_{\bF_0}((a-1)f+(2b-1)f')
\quad\text{and}\quad
\cO_{\bF_0}((2a-1)f+(b-1)f')
\]
are the Ulrich line bundles with respect to $D$.
\end{example}

\section{Projective bundles over surfaces}

In this section, we specialise to the question about Ulrich bundles on projective bundles over surfaces. Later we pay more attention to the case of the projective plane and Hirzebruch surfaces.

\begin{theorem}
\label{thm:bundles-surfaces}
Let $\pi \colon \bP(\cE) \to S$ be a projective bundle over a smooth projective surface $S$.
Moreover, let $D = \pi^* A + H$ be a very ample divisor with $H$ the relative hyperplane section. 

A locally free sheaf $\pi^* \cF(D)$ is Ulrich on $\bP(\cE)$ with respect to $D$ if and only if $\cF$ is a locally free sheaf on $S$ with 
\[
\sH^\bullet(S,\cF) = 0 = \sH^\bullet(S,\cF(-D'))
\]
where $D' \coloneqq \rank(\cE)\cdot A + c_1(\cE)$.
In particular, if $D'$ is very ample, then $\cF(D')$ is Ulrich with respect to $D'$.
\end{theorem}

This statement generalises \cite[Thm. 2.4]{Fania-etal}, where the case of $\bP^1$-bundles over surfaces is treated.
Recall that by \autoref{lem:bundles-veryample}, a very ample divisor of the form $D = \pi^*A + H$ does always exist.

\begin{proof}
This is a direct application of \autoref{prop:bundles-general}.
\end{proof}

\begin{question}
\label{q:very-ampleness-bundle-base}
Already the connection between the ample divisors on a surface and ample divisors on projective bundles over the surface is quite complicated, see \cite{Misra-Ray}.
So it is not obvious, whether $D = \pi^*A + H$ (very) ample implies that also $D' = (n+1)A + c_1(\cE)$ is (very) ample, as well.
Examples suggest that this implication is true, but also show that the other implication is wrong in general; see \autoref{ex:blowupP3}.

On the other hand, by \autoref{lem:bundles-veryample}, we can choose $A$ positive enough such that $D$ and $D'$ are very ample.
\end{question}

\begin{example}
In \cite{Fania-etal}, many more Ulrich bundles of low rank are constructed on $\bP^1$-bundles over several surfaces, obtaining even a full classification in those cases of Ulrich line bundles.
\end{example}

\subsection{Projective bundles over $\bP^2$}

\begin{proposition}[{\cite{Coskun-Genc, Costa-Miro-Roig}}]
\label{prop:ulrich-p2}
For $d>0$, there are are Ulrich bundles on $\bP^2$ with respect to $\cO_{\bP^2}(d)$:
\begin{itemize}
\item[$d=1$:] direct sums of $\cO_{\bP^2}$;
\item[$d=2$:] direct sums of $\cT_{\bP^2}$;
\item[$d>2$:] locally free sheaves of rank at least $2$. For example, an Ulrich bundle $\cE$ of rank two fits into the short exact sequence
\[
0 \to \cO_{\bP^2}^{\oplus (d-1)}(d-2) \to \cO_{\bP^2}^{\oplus (d+1)}(d-1) \to \cE \to 0,
\]
which in the case of $d=2$ yields $\cT_{\bP^2}$.
\end{itemize}
\end{proposition}

The following lemma will be the basis for the construction of Ulrich bundles.

\begin{lemma}
\label{lem:XXX}
Let $d$ be a non-negative integers.
On $\bP^n$, there are locally free sheaves $\cF$ of rank $n$ or ${n+d \choose n-1}$ with
the properties that
\[
\sH^\bullet(\bP^n,\cF) = 0
\quad \text{and} \quad
\cF \in \sod{ \cO_{\bP^n}(d), \cO_{\bP^n}(d+1)}.
\]
\end{lemma}

\begin{proof}
The easiest way to guarantee the second condition of the statement is to assume that $\cF$ fits into a short exact sequence
\[
0 \to \cF \to \cO(d)^{\oplus b_1} \xto{\alpha} \cO(d+1)^{\oplus b_2} \to 0,
\]
as for a surjective morphism between locally free sheaves, the kernel is automatically locally free.
By applying $\sH^\bullet(\bP^n,\blank)$ to this sequence, we obtain
\[
\begin{split}
0 &\to \sH^0(\bP^n,\cF) \to \sH^0(\bP^n,\cO(d))^{\oplus b_1} \xto{\sH^0(\alpha)} \sH^0(\bP^n,\cO(d+1))^{\oplus b_2} \to \\
&\to \sH^1(\bP^n,\cF) \to 0.
\end{split}
\]
Here we want to ensure that $\sH^0(\alpha)$ is an isomorphism to conclude that $\sH^\bullet(\bP^n,\cF) = 0$.
A necessary condition for this to happen is that
$b_1 \cdot \dim \sH^0(\bP^n,\cO(d)) = b_2 \cdot \dim \sH^0(\bP^n,\cO(d+1))$, or more explicitly that
\[
b_1 \cdot {n+d \choose n} = b_2 \cdot {n+d+1 \choose n}
\quad \iff \quad
(d+1) b_1 = (n+d+1) b_2
\]

In particular, $b_1 = n+d+1$ and $b_2 = d+1$ is a valid solution.
There are solutions with smaller values, if $\gcd(n,d+1) > 1$,
but for this choice we find that there a generic choice of
$\alpha \colon \cO(d)^{n+d+1} \to \cO(d+1)^{d+1}$
induces an isomorphism $\sH^0(\alpha)$. For example, we can take
\[
\cO^{n+d+1}(d) 
\xto{
\left(
\def\dots[#1]{ \ar[#1, dotted, no head, shorten=1mm]}
\begin{tikzcd}[row sep=tiny, column sep=small, ampersand replacement=\&]
x_0 \dots[rrrddd] \dots[rr] \&\& x_n \dots[rrrddd] \& 0 \dots[rr] \dots[rrdd] \&\& 0 \dots[dd] \\
0 \dots[rrdd] \dots[dd]\\ 
 \& \&\& \&\& 0 \\
0 \dots[rr] \&\& 0 \& x_0 \dots[rr] \&\& x_n
\end{tikzcd}
\right)
}
\cO^{d+1}(d+1) 
\]
which is a direct generalisation of \cite[Prop.\ 5.9]{Eisenbud-Schreyer-Weyman}.
Hence for $b_1 = n+d+1$ and $b_2 = d+1$, we obtain a locally free sheaf $\cF$ of rank $n$ with
\[
\sH^\bullet(\bP^n,\cF) = 0, \quad
\cF \in \sod{ \cO_{\bP^n}(d), \cO_{\bP^n}(d+1)}.
\]

Another choice for a solution is
$b_1 = {n+d+1 \choose n}$ and $b_2 = {n+d \choose n}$.
By the same kind of argument, we obtain for this choice,
a locally free sheaf $\cF$ of rank $b_1-b_2 = {n+d \choose n-1}$.
Here, we can construct also an explicit locally free sheaf.
For this, consider the Euler sequence
\[
0 \to \Omega \to \cO(-1)^{n+1} \to \cO \to 0.
\]
Taking the $(d+1)$-th symmetric power of this short exact sequence and tensoring with $\cO(2d+1)$ yields
\[
0 \to \sym^{d+1} \Omega(2d+1) \to \cO(d)^{\oplus {n+d+1 \choose n}} \to \cO(d+1)^{\oplus {n+d \choose n}} \to 0.
\]
By construction $\sym^{d+1} \Omega(2d+1) \in \sod {\cO(d),\cO(d+1)}$, and taking cohomology of the last short exact sequence shows also that $\sH^\bullet(\bP^n, \sym^{d+1} \Omega(2d+1))=0$.
\end{proof}

\begin{remark}
\label{ex:ulrich-p2}
Let $\cF$ be a locally free sheaf on $\bP^2$.
The conditions of \autoref{lem:XXX} amount to $\cF(d+2)$ being an Ulrich bundle with respect to $\cO(d+2)$. To see this, note that $0 = \sH^\bullet(\bP^2,\cF(-d-2)) = \Hom^\bullet(\cO(d+2),\cF)$ if and only if $\cF \in \cO(d+2)^\perp = \sod{\cO(d),\cO(d+1)}$.

On $\bP^2$, note that in the case $d=0$, we obtain the Ulrich bundle 
\[
\cF(2) \cong \Omega_{\bP^2}(3) \cong \cT_{\bP^2}
\]
of rank $2$ on $\bP^2$.
Actually, this is the only indecomposable Ulrich bundle with respect to $\cO_{\bP^2}(2)$, see \cite[Ex. 3.1]{Aprodu-HML}.
Moreover, we get for any $d > 0$ many Ulrich bundles of rank two. The construction in the proof of \autoref{lem:XXX} is the dual of the one given in \cite[Prop.\ 5.9]{Eisenbud-Schreyer-Weyman}.

Also note that the argument of \autoref{lem:XXX} can be extended to cover also the case $d=-1$, where we end up with the the Ulrich line bundle $\cO_{\bP^2}$.
\end{remark}

\begin{remark}
On $\bP^2$, in the case $d=1$, the proof of \autoref{lem:XXX} gives the locally free sheaf $\cF = \sym^2 \Omega_{\bP^2}(3)$, which appears in \cite[\S 3.4]{Hochenegger-Meachan} in a completely different context.
There the interest lies on ${}^\perp \cO_{\bP^2} \cap \cO_{\bP^2}^\perp$, which is a non-admissible subcategory of $\Db(\bP^2)$.
Note that the locally free sheaves inside ${}^\perp \cO_{\bP^2} \cap \cO_{\bP^2}^\perp$ are precisely the Ulrich bundles on $\bP^2$ with respect to the anticanonical polarisation $\cO_{\bP^2}(3)$.
\end{remark}

The following statement is a consequence of \autoref{thm:bundles-surfaces} and \autoref{lem:XXX}, by taking into account that for $D = \pi^* A + H$ we can replace $A$ by a more positive divisor, in order to guarantee that $D' = \rank(\cE)\cdot A + c_1(\cE)$ is very ample, see \autoref{lem:bundles-veryample}.

\begin{corollary}
\label{cor:P2}
Let $\pi \colon \bP(\cE) \to \bP^2$ be a projective bundle and let $D = \pi^*A+H$ be a very ample divisor such that $D' = \rank(\cE)\cdot A + c_1(\cE)$ is very ample, as well.
Then there are Ulrich bundles of rank two on $\bP(\cE)$ with respect to $D$.
\end{corollary}

\begin{remark}
This corollary generalises the results from \cite[\S 5]{Fania-etal} about Ulrich bundles of rank two on $\bP^1$-bundles over $\bP^2$.
Note that in \cite[Prop. 5.1]{Fania-etal}, it is shown under which restrictions on the Chern classes of $\cE$, there are Ulrich line bundles.
\end{remark}

\begin{question}
This is again \autoref{q:very-ampleness-bundle-base}:
If $D = \pi^*A +H$ is very ample, do we have at least in the situation of the theorem above, that $D' = (n+1)A+c_1(\cE)$ is very ample?
What is the minimal degree $d_{\min}$ of $D'$ in dependence of the possible choices for $A$?
In the case that $d_{\min}$ is always positive, we have a connection between Ulrich bundles on $\bP(\cE)$ and those on $\bP^2$.
In situations where $d_{\min}=1$ is possible, $\pi^*\cO_{\bP^2}(-1) \otimes \cO_{\bP(\cE)}(D)$ would be an Ulrich line bundle on $\bP(\cE)$.
\end{question}

\begin{example}
\label{ex:blowupP3}
Let $X = (\bP^3)'$ be the blowup of $\bP^3$ in a point.
There is a map $\pi \colon X \to \bP^2$, realising $X$ as a projective bundle.
More precisely, one has $X = \bP(\cE)$ with $\cE = \cO_{\bP^2}(1) \oplus \cO_{\bP^2}$. 
Denoting by $h$ the pullback of a hyperplane section of $\bP^2$ via $\pi$, a divisor $D = ah+H$ is very ample if and only if $a > 0$.

By \autoref{thm:bundles-surfaces}, a locally free sheaf $\pi^*\cF(ah+H)$ is Ulrich with respect to $\cO_X(ah+H)$ if and only if $\cF(2a+1)$ is Ulrich on $\bP^2$ with respect to $\cO_{\bP^2}(2a+1)$.
For the minimal choice $a=1$, we therefore have Ulrich bundles of rank at least two, see \autoref{ex:ulrich-p2}.

Note that for the excluded value $a=0$, $\cO_{\bP^2}$ is an Ulrich line bundle on $\bP^2$ with respect to the very ample $\cO_{\bP^2}(1)$.
Moreover, for $\cF = \cO_{\bP^2}(-1)$, the cohomolgy groups $H^\bullet(X,\cL(-pD))$ vanish for the line bundle $\cL \coloneqq \pi^*\cF(H) = \cO_X(H-h)$ and $p=1,2,3$. 
But $\cO_X(H)$ is not even ample, so $\cO_X(H-h)$ is not Ulrich.

For a deeper picture about Ulrich bundles on $(\bP^3)'$, see \cite{Casnati-etal}.
\end{example}

\begin{example}
\label{ex:blowupP3b}
Consider $\cE = \cO_{\bP^2}(1) \oplus \cO_{\bP^2}^{\oplus d}$ and 
the projective bundle $\pi \colon \bP(\cE) \to \bP^2$.
As before, by \autoref{thm:bundles-surfaces}, 
a locally free sheaf $\pi^*F(ah + H)$ is Ulrich with respect to $ah + H$ if and only if $\cF((d + 1)a + 1)$ is Ulrich on $\bP^2$ with respect to $\cO_{\bP^2}((d + 1)a + 1)$. 
Again, the case $a = 0$ is not possible, as $H$ is not even ample, 
so we conclude that there are Ulrich bundles on X which have rank at least two.
\end{example}

\marginpar{{\small The proof  of \autoref{prop:ulrich-bundle-pn} contains an error, see the Erratum on page \pageref{Erratum} for further details.}}
\begin{proposition}
\label{prop:ulrich-bundle-pn}
Let $\cE = \cO_{\bP^n}(1) \oplus \cO_{\bP^n}^{\oplus d}$ and consider the projective bundle $\pi \colon \bP(\cE) \to \bP^n$.
Then there are Ulrich bundles of rank $n$ on $\bP(\cE)$ with respect to 
the very ample divisor $D = h+H$, where $h$ is the pullback of the hyperplane section of $\bP^n$.
\end{proposition}

\begin{proof}
We show that there is a locally free sheaf $\cF$ of rank $n$ on $\bP^n$,
such that $\pi^*\cF(D)$ is an Ulrich bundle on $\bP(\cE)$ with respect to $D$.
For this we will use \autoref{prop:bundles-general}
in its formulation for split locally free sheaves in \autoref{rem:split-bundle}.
The conditions for $\cF$ to fulfill are that $\sH^\bullet(\bP^n,\cF)=0$ and
$\sH^\bullet(\bP^n,\cF(-1-\delta-(n+1+k))) = 0$
with $\delta \in \{0,1\}$ and $0 \leq k \leq n-2$.
The latter vanishing can be expressed as
\[
\Hom^\bullet(\cO(n+2+k),\cF) = 0
\]
for $0 \leq k \leq n-1$.
Using that $\Db(\bP^n)=\sod{\cO(n),\cO(n+1),\cO(n+2),\ldots,\cO(2n+1)}$,
these conditions become
\[
\sH^\bullet(\bP^n,\cF) = 0
\quad\text{and}\quad
\cF \in \sod{\cO(n),\cO(n+1)}.
\]
So by \autoref{lem:XXX} there are locally free sheaves $\cF$ of rank $n$ on $\bP^n$ that satisfy these conditions.
In turn, this gives Ulrich bundles $\pi^*F(D)$ of rank $n$ on $\bP(\cE)$ with respect to $D=h+H$ by \autoref{prop:bundles-general}.
\end{proof}

\begin{remark}
\label{rem:blowup}
For $\cE = \cO_{\bP^n}(1) \oplus \cO_{\bP^n}^{\oplus d}$, the projective bundle $\bP(\cE)$ is isomorphic to the blowup $(\bP^{n+d})'$ of $\bP^{n+d}$ in a $(d-1)$-dimensional linear subspace.
So in \autoref{ex:blowupP3} \& \ref{ex:blowupP3b} and \autoref{prop:ulrich-bundle-pn},
there is also the possibility to construct Ulrich bundles via pullback using the map $(\bP^{n+d})' \to \bP^{n+d}$,
by using \cite[Thm. 3.1]{Casnati-Kim}.
\end{remark}

\subsection{Projective bundles over Hirzebruch surfaces}

\begin{proposition}[\cite{Aprodu-etal}]
Let $\pi \colon \bF_r \to \bP^1$ be the projection of Hirzebruch surface $\bF_r = \bP(\cO \oplus \cO(r))$ to $\bP^1$.
For any very ample $D = ah + bH$, where $h$ is the pullback of the hyperplane section of $\bP^1$, there exist
\begin{itemize}
\item Ulrich line bundles if and only if $b=1$,
\item Ulrich bundles of rank two if $b>1$
\end{itemize}
on $\bF_r$ with respect to $D$.
\end{proposition}

\begin{remark}
In \cite{Antonelli-Hirzebruch}, Ulrich bundles of arbitrary rank are described as cokernels of an injective map of split locally free sheaves.
\end{remark}

Again as an application of \autoref{thm:bundles-surfaces}, we arrive at the following statement.

\begin{corollary}
\label{cor:hirzebruch}
Let $\pi \colon \bP(\cE) \to \bF_r$ be a projective bundle and let $D = \pi^*A+H$ be a very ample divisor such that $D' = \rank(\cE)\cdot A + c_1(\cE)$ is very ample, as well.
Then there are Ulrich bundles of rank two on $\bP(\cE)$ with respect to $D$.
\end{corollary}

\begin{remark}
For the base variety $\bF_0 = \bP^1 \times \bP^1$, there are Ulrich line bundles for any polarisation, so in this case we have even Ulrich line bundles on $\bP(\cE)$.

For base variety $\bF_r$ with $r>0$, the question remains under which conditions on $\cE$ there are Ulrich line bundles on $\bP(\cE)$.
Note that in \cite[Thm. 5.9]{Fania-etal}, it is shown that there are Ulrich bundles of rank two on a $\bP^1$-bundle over $\bF_1$.
\end{remark}

\begin{remark}
\label{rem:further-surfaces}
The method described in this section, establishes the existence of Ulrich bundles on projective bundles over a surface, where sufficiently many Ulrich bundles are known. 
We give a list of such surfaces, which is far from exhaustive, and some references:
\begin{itemize}
\item del Pezzo surfaces: \cite{Pons-Tonini, Coskun-Pezzo};
\item K3 surfaces: \cite{Coskun-K3, Watanabe, Aprodu-Farkas-Ortega, Casnati-Galuzzi};
\item surfaces with $p_q = 0$ and $q=0,1$: \cite{Casnati-q0, Casnati-q1}.
\end{itemize}
Moreover, most of the results of \cite{Fania-etal} about Ulrich bundles on $\bP^1$-bundles over certain surfaces should allow a generalisation to arbitrary projective bundles over these surfaces.
We leave the search for further Ulrich bundles to the interested readers.
\end{remark}

\newpage
\section*{Erratum}
\label{Erratum}

Unfortunately, there is an error in the proof of \autoref{prop:ulrich-bundle-pn}, which we recall here:

\begin{statement}{Proposition 5.13}[Wrong]
Let $\cE = \cO_{\bP^n}(1) \oplus \cO_{\bP^n}^{\oplus d}$ and consider the projective bundle $\pi \colon \bP(\cE) \to \bP^n$.
Then there are Ulrich bundles of rank $n$ on $\bP(\cE)$ with respect to
the very ample divisor $D = h+H$, where $h$ is the pullback of the hyperplane section of $\bP^n$.
\end{statement}

In the course of the proof, we have wanted to establish the existence of a locally free sheaf $\cF$ of rank $n$ that satisfies the conditions 
\begin{equation}
\tag{{$\ast$}}
\label{eq:conditions}
\sH^\bullet(\bP^n,\cF) = 0
\quad\text{and}\quad
\Hom^\bullet(\cO(n+2+k),\cF) = 0
\end{equation}
for $0 \leq k \leq n-1$.
From this one can conclude by \autoref{prop:bundles-general} that $\pi^*\cF(D)$ is an Ulrich bundle on $\bP(\cE)$ with respect to $D$.

At that point, we have erroneously claimed that $\Db(\bP^n)=\sod{\cO(n),\cO(n+1),\cO(n+2),\ldots,\cO(2n+1)}$, but one has to remove the first (or last) invertible sheaf to obtain a semi-orthogonal decomposition of $\Db(\bP^n)$.
By correctly using that $\Db(\bP^n)=\sod{\cO(n+1),\cO(n+2),\ldots,\cO(2n+1)}$,
the conditions of \eqref{eq:conditions} become
\[ 
\sH^\bullet(\bP^n,\cF) = 0
\quad\text{and}\quad
\cF \in \sod{\cO(n+1)}.
\]
But these conditions can only be satisfied for $\cF=0$, as any locally free sheaf in $\sod{\cO(n+1)}$ is of the form $\cO(n+1)^{\oplus m}$ for some $m$, so has non-zero global sections if $m>0$.

\medskip
\begin{center}
Hence \autoref{prop:ulrich-bundle-pn} as stated above is \emph{wrong}.
\end{center}
\bigskip

There are two ways to correct the statement using the same line of arguments, both yielding statements of less interest. On the one hand, we observe that the corrected proof shows the non-existence of a certain kind of Ulrich bundles on $\bP(\cE)$:

\begin{statement}{Proposition 5.13.A}
Let $\cE = \cO_{\bP^n}(1) \oplus \cO_{\bP^n}^{\oplus d}$ and consider the projective bundle $\pi \colon \bP(\cE) \to \bP^n$.
Moreover, let $D = h+H$ be the very ample divisor on $\bP(\cE)$, where $h$ is the pullback of the hyperplane section of $\bP^n$.
Then there is no Ulrich bundle on $\bP(\cE)$ with respect to $D$ of the form $\pi^*\cF(D)$, where $\cF$ is a locally free sheaf on $\bP^n$.
\end{statement}

On the other hand, the problem arises because the range for the index $k$ in \eqref{eq:conditions} is slightly too big. 
So in order to apply \autoref{rem:split-bundle}, one may look at $\cE = \cO_{\bP^n}^{\oplus (d+1)}$ and $\bP(\cE) = \bP^n \times \bP^d$. In this situation the conditions in \eqref{eq:conditions} are the same except that $0 \leq k \leq n-2$. Continuing with the same line of arguments as in the proof of \autoref{prop:ulrich-bundle-pn}, one arrives at a point where one can apply \autoref{lem:XXX}. Therefore, we obtain:

\begin{statement}{Proposition 5.13.B}
Let $\cE = \cO_{\bP^n}^{\oplus (d+1)}$ and consider the projective bundle $\pi \colon \bP(\cE) = \bP^n \times \bP^d \to \bP^n$.
Then there are Ulrich bundles of rank $n$ on $\bP(\cE)$ with respect to
the very ample divisor $D = h+H$, where $h$ is the pullback of the hyperplane section of $\bP^n$.
\end{statement}

If $n>1$, then this statement certainly does not yield Ulrich bundles of the lowest possible rank: there are Ulrich bundles of rank 1 for this polarization, see \cite[(3.5)]{beauville}.

\subsection*{Acknowledgements}
The author wants to thank Roberto Vacca for pointing out the error in the proof of \autoref{prop:ulrich-bundle-pn}.

\let\oldthebibliography=\thebibliography
\let\oldendthebibliography=\endthebibliography
\renewenvironment{thebibliography}[1]{%
    \oldthebibliography{#1}%
    \setcounter{enumiv}{28}%
}{\oldendthebibliography}

\BgThispage

\end{document}